\documentclass[12pt]{article}

\usepackage{amssymb,amsthm,amsmath}
\usepackage[utf8]{inputenc}
\usepackage{color}
\usepackage{graphicx}
\usepackage{geometry}

\newcommand{\R}{\mathbb{R}}

\newcommand{\1}{\textbf{1}}
\newcommand{\dd}{\mathrm{d}}

\newcommand{\e}{\varepsilon}

\newtheorem{thm}{Theorem}

\theoremstyle{definition}
\newtheorem*{question*}{Question}

\begin{document}

\newgeometry{tmargin=2.5cm, bmargin=2.5cm, lmargin=2.5cm, rmargin=2.5cm}

\title{Reverse isoperimetric inequalities for parallel sets}

\author{Piotr Nayar
\thanks{The author was supported  by the National Science Centre, Poland, grant 2018/31/D/ST1/01355.}
}

\maketitle

\abstract{We consider the family of $r$-parallel sets in $\R^d$, that is sets of the form $A_r=A+rB_2^n$, where $B_2^n$ is the unit Euclidean ball and $A$ is arbitrary Borel set.
We show that the ratio between the upper surface area measure of an $r$-parallel set and its volume is upper bounded by $d/r$. Equality is achieved for $A$ being a single point. 

As a consequence of our main result we show that the Gaussian upper surface area measure of an $r$-parallel set is upper bounded by $18d \max(\sqrt{d},r^{-1})$. Moreover, we observe that there exists a $1$-parallel set with Gaussian surface area measure at least $0.28 \cdot d^{1/4}$.}

\vspace{1cm}

\noindent 2010 \emph{Mathematics Subject Classification}: Primary 52A40; Secondary 60G15.

\section{Introduction}

For sets $A,B$ is $\R^d$ we define their Minkowski sum $A+B=\{a+b: \ a \in A, b \in B\}$. Suppose $K$ is some compact convex set and let $r>0$. A set of the form $A_{r,K}=A+rK$ is called $(r,K)$-parallel. If $K=B_2^d$ is a unit Euclidean ball then $A_{r,K}$ will simply be called  $r$-parallel and will be denoted by $A_r$.  For a Borel set $A$ we shall write
\[
	|\partial A|_+ = \limsup_{\e \to 0^+} \frac{|A_\e| - |A|}{\e}, \qquad |\partial_K A|_+ = \limsup_{\e \to 0^+} \frac{|A_{\e,K}| - |A|}{\e},
\]  
where $|\cdot|$ stands for the Lebesgue measure. The quantity $|\partial A|_+$ is called the upper surface area measure of $A$.

  In \cite{J20} Jog considered reverse isoperimetric inequalities for parallel sets. He proved that for any compact set $A$ in $\R^d$ one has 
$
 	|\partial A_r| \leq d 2^{2d-1} \frac{|A_r|}{r}. 
$  
In this note we prove the following sharp result.

\begin{thm}\label{thm:1}
Let $A$ be a Borel set and let $K$ be compact and convex. Then  
\[
	|\partial_K A_{r,K}|_+ \leq \frac{d}{r} \cdot |A_{r,K}|,
\]
which is tight for $A=\{0\}$. In particular 
$
	|\partial A_{r}|_+ \leq \frac{d}{r} \cdot |A_{r}|,
$ .
\end{thm}

In \cite{J20} also the Gaussian case was treated. The upper Gaussian surface area of a measurable set is defined as
\[
\gamma_d^+(\partial A) = \limsup_{\e \to 0^+} \frac{\gamma_d(A_\e) - \gamma_d(A)}{\e},
\]
where $\gamma_d$ stands for the standard Gaussian measure, that is measure with  density $(2\pi)^{d/2} e^{-|x|^2/2}$. Jog proved the inequality $\gamma_d^+(\partial A_r) \leq 2^{2d-1}d^2 3^d \max(1,r^{-1})$. We shall prove this bound with a better dependence on the dimension.

\begin{thm}\label{thm:2}
For any Borel set  $A$ we have 
\[
	\gamma_d^+(\partial A_r) \leq 18d \max(\sqrt{d},r^{-1}).
\]
\end{thm}

We mention that it is not possible to remove the dimension dependence in the above estimate: there exists a $1$-parallel set whose Gaussian surface area is of order $d^{1/4}$. This follows from a simple observation: every set of the form $K^c$ (complement of $K$), where $K$ is open and convex, is $r$-parallel for every $r>0$. It is easy to verify that any closed halfspace $H$ is $r$-parallel for every $r>0$ (in fact $H=A_r$ for an appropriate halhspace $A$). Since every open convex set $K$ is of the form $K = \bigcap_{i \in I} H_i$ for some family of open halfspaces $(H_i)_{i \in I}$, we have 
\[
	 K^c = \bigcup_{i \in I} H_i^c = \bigcup_{i \in I} (A_i)_r  = \Big( \bigcup_{i \in I} A_i \Big)_r,   
\]   
where sets $A_i$ satisfy $(A_i)_r = H_i^c$ (note that $H_i^c$ are closed halfspaces). According to the result of Nazarov from \cite{N03}, there exists a convex set $K$ such that $\gamma_d^+(\partial K) \geq 0.28 \cdot d^{1/4}$.

\section{Proofs}

We first prove Theorem \ref{thm:1}.

\begin{proof}[Proof of Theorem \ref{thm:1}]
According to the result of Fradelizi and Marsiglietti from \cite{FM14} (Proposition 2.1), for any compact set $A$ in $\R^d$ and any compact convex set $K$ in $\R^d$ the function
\[
	(s,t) \ \longmapsto \ |sA+tK|
\]
is non-decreasing on $\R_+ \times \R_+$ in each coordinate. This is a consequence of the result of Stach\'o from \cite{S76}. Since for $r>0$ we have
$
	\frac{|A+rK|}{r^d} = |r^{-1}A + K|, 
$ 
the left hand side is non-increasing. Thus, for any $\e>0$ we have
\begin{align*}
	0 \geq & \frac{1}{\e}\left(\frac{|A+(r+\e)K|}{(r+\e)^d} - \frac{|A+rK|}{r^d} \right)  \\
	& = \frac{1}{(r+\e)^d} \cdot \frac{|A+(r+\e)K|- |A+rK|}{\e} + |A+rK| \cdot  \frac{\frac{1}{(r+\e)^d} - \frac{1}{r^d}}{\e}. 
\end{align*}
Taking $\e \to 0^+$ we arrive at
\[
	0 \geq \frac{|\partial_K A_{r,K}|_+}{r^d} - \frac{d}{r^{d+1}} |A_{r,K}|.
\]
The proof is completed.
\end{proof}

Theorem \ref{thm:1} implies Theorem \ref{thm:2}. For the proof we follow the strategy developed in \cite{J20}.

\begin{proof}[Proof of Theorem \ref{thm:2}]
Through the proof $c$ is a universal constant independent of the dimension, whose value may change from one line to the next. Note that for a measurable set $A$ we have
\begin{align*}
	\gamma_d(A) & = (2\pi)^{-\frac{d}{2}}\int_A e^{-|x|^2/2} \dd x  = (2\pi)^{-\frac{d}{2}}\int_A \int_{|x|}^\infty t e^{-t^2/2} \dd t \dd x \\
	& =    (2\pi)^{-\frac{d}{2}} \int_0^\infty \int_{\R^d}  t e^{-t^2/2} \1_{\{|x| \leq t\}} \1_A  \dd x \dd t  = (2\pi)^{-\frac{d}{2}} \int_0^\infty   t e^{-t^2/2} |A \cap t B_2^d | \dd t.
\end{align*}
Let us fix $\e_0>0$ and take $0<\e<\e_0$. Let $A^t=A \cap (t+r+\e_0)B_2^d$. We have
\begin{align*}
	\gamma_d(A_{r+\e})-\gamma_d(A_r) & = (2\pi)^{-\frac{d}{2}} \int_0^\infty   t e^{-t^2/2} \ |(A_{r+\e} \setminus A_r) \cap t B_2^d| \ \dd t \\
	& = (2\pi)^{-\frac{d}{2}} \int_0^\infty   t e^{-t^2/2} \ |((A^t)_{r+\e} \setminus (A^t)_r) \cap t B_2^d| \ \dd t \\
	& \leq (2\pi)^{-\frac{d}{2}} \int_0^\infty   t e^{-t^2/2} \ |(A^t)_{r+\e} \setminus (A^t)_r | \ \dd t.
\end{align*}
Dividing by $\e$, taking the limit $\e \to 0^+$ and applying Theorem \ref{thm:1} gives
\begin{align*}
	\gamma_d^+(\partial A_r) & \leq (2\pi)^{-\frac{d}{2}} \int_0^\infty   t e^{-t^2/2} \ |\partial (A^t)_r |_+ \ \dd t \leq (2\pi)^{-\frac{d}{2}}  \cdot \frac{d}{r}\int_0^\infty   t e^{-t^2/2} \ |(A^t)_r | \ \dd t \\
	& \leq (2\pi)^{-\frac{d}{2}}  \cdot \frac{d}{r}\int_0^\infty   t e^{-t^2/2} \ |(t+2r+\e_0) B_2^d | \ \dd t \\ & = \frac{|B_2^d|}{(2\pi)^{\frac{d}{2}}}  \cdot \frac{d}{r}\int_0^\infty   t e^{-t^2/2} \ (t+2r+\e_0)^d \dd t.  
\end{align*}
Taking the limit $\e_0 \to 0^+$ yields
\[
	\gamma_d^+(\partial A_r) \leq \frac{|B_2^d|}{(2\pi)^{\frac{d}{2}}}  \cdot \frac{d}{r}\int_0^\infty   t e^{-t^2/2} \ (t+2r)^d \dd t.  
\]
For $p>-1$ we have 
\[
	m_{p}:=\int_0^\infty e^{-t^2/2} t^p \dd t = 2^{\frac{p-1}{2}} \Gamma \left(\frac{p+1}{2}\right).
\]  
We also have $|B_2^d| = \frac{\pi^{d/2}}{\Gamma(\frac{d}{2}+1)}$. As a consequence $m_{d+1} = 2^{\frac{d}{2}} \Gamma \left(\frac{d}{2}+1\right) = (2\pi)^{\frac{d}{2}}/|B_2^d|$ and thus
\begin{align*}
	\gamma_d^+(\partial A_r) & \leq \frac{d}{r} \cdot \frac{1}{m_{d+1}} \cdot \int_0^\infty   t e^{-t^2/2} \ (t+2r)^d \dd t = \frac{d}{r} \cdot \frac{1}{m_{d+1}} \cdot   \sum_{i=0}^d {d \choose i} m_{i+1} (2r)^{d-i}  \\
	& = \frac{d}{r}  \cdot   \sum_{i=0}^d {d \choose i} \frac{\Gamma(\frac{i}{2}+1)}{\Gamma(\frac{d}{2}+1)} 2^{\frac{i-d}{2}}  (2r)^{d-i} .
\end{align*}   
 Using the standard bounds 
 \[
 	\sqrt{2\pi x} x^x e^{-x} \leq \Gamma(x+1) \leq 2\sqrt{2\pi x} x^x e^{-x}, \qquad x \in [1,\infty) \cup \{1/2\}.
 \] 
 and ${d \choose i} \leq \frac{d^{d-i}}{(d-i)!}$
  we get
  \begin{align*}
   \sum_{i=0}^d {d \choose i} \frac{\Gamma(\frac{i}{2}+1)}{\Gamma(\frac{d}{2}+1)} 2^{\frac{i-d}{2}}  (2r)^{d-i} & \leq  2 \sum_{i=0}^d \frac{d^{d-i}}{(d-i)!}  \cdot \frac{(i/2)^{i/2} e^{-i/2}}{(d/2)^{d/2} e^{-d/2}} (\sqrt{2} r)^{d-i} \\
   & = 2 \sum_{i=0}^d \frac{d^{d-i}}{(d-i)!}  \cdot \frac{i^{i/2} }{d^{d/2}} (2 \sqrt{ e} r)^{d-i}.  
  \end{align*}
Let us now assume that $r \leq r_\ast := \frac{d^{-1/2}}{2 \sqrt{e}}$. Then
\begin{align*}
2 \sum_{i=0}^d \frac{d^{d-i}}{(d-i)!}  \cdot \frac{i^{i/2} }{d^{d/2}} (2 \sqrt{ e} r)^{d-i} & \leq 2 \sum_{i=0}^d \frac{d^{d-i}}{(d-i)!}  \cdot \frac{i^{i/2} }{d^{d/2}} \cdot  \frac{1}{d^{\frac{d-i}{2}}} \\ &   = 2 \sum_{i=0}^d \frac{1}{(d-i)!} \cdot \left(\frac{i}{d}\right)^{i/2}  \leq  2 \sum_{i=0}^d \frac{1}{(d-i)!} \leq 2e .  
\end{align*}
Therefore, for $r \leq r_\ast$ we get $\gamma_d^+(\partial A_r) \leq \frac{2de}{r} $.

If $r>  r_\ast$ one can use the bound for $r = r_\ast$, since every $r$-parallel set is $r'$ parallel for every $r'<r$. Thus in this case we get $\gamma_d^+(\partial A_r) \leq 4 e^{3/2} d^{3/2} < 18 d^{3/2}$.  We proved that always $\gamma_d^+(\partial A_r) \leq \max(18 d^{3/2},\frac{2de}{r}) \leq 18d \max(\sqrt{d},\frac{1}{r})$.  
\end{proof}

\vspace{1cm}  
  
\noindent Institute of Mathematics \\ University of Warsaw\\ Banacha 2, 02-097, Warsaw, Poland \\
email: nayar@mimuw.edu.pl

\end{document}